\documentclass[reqno]{amsart}
\usepackage{amsfonts}
\usepackage{graphicx}
\usepackage{amscd}
\usepackage{amsmath}
\usepackage{amssymb}
\usepackage{latexsym}
\usepackage[top=2.54cm,bottom=2.0cm,left=2.54cm,right=2.0cm]{geometry}
\usepackage[pdftex, colorlinks = true, urlcolor=black, citecolor=blue, linkcolor = black]{hyperref} 
\usepackage[all]{xy}
\theoremstyle{plain}

\newtheorem{corollary}{\bf Corollary}

\newtheorem{eexample}{\bf Example}
\newtheorem{lemma}{\bf Lemma}

\newtheorem{remark}{\bf Remark}

\newtheorem{theorem}{\bf Theorem}
\numberwithin{equation}{section}

\begin{document}

\title{A Simons' type formula for spacelike submanifolds in semi-Riemannian warped product}

\author{Guillermo A. Lobos$^1$}
\address{$^2$Departamento de Matem\'atica-UFScar, 13565-905 - São Carlos-SP-Brazil}
\email{lobos@ufscar.br}
\urladdr{http://www.dm.ufscar.br/profs/lobos}

\author{Mynor Melara$^2$}
\address{$^3$Departamento de Matem\'atica-UFScar, 13565-905 - São Carlos-SP-Brazil}
\email{mynormelara@estudante.ufscar.br}
\urladdr{http://www.ufscar.br}

\author{Maria R. B. Santos$^3$}
\address{$^1$Departamento de Matem\'atica-ICE-UFAM, 69080-900 - Manaus-AM-BR}
\email{mrosilenesantos@ufam.edu.br}
\urladdr{http://www.ufam.edu.br}

\begin{abstract}
We determine a Simons' type formula for spacelike submanifolds within a broad class of semi-Riemannian warped products. This formula extends the Simons' type formulas initially introduced by Nomizu and Smyth in 1969 for constant mean curvature hypersurfaces in space forms. Furthermore, it incorporates the 2013 extension by Fetcu and Rosenberg for submanifolds with parallel mean curvature vector field in product spaces. From a global standpoint, we establish that compact spacelike hypersurfaces with parallel mean curvature and non-negative sectional curvature are isoparametric hypersurfaces. This result constitutes a generalization of the Riemannian case within space forms, as demonstrated by Nomizu and Smyth. From a local standpoint,  
we extend analogous results previously established by Asperti, Lobos and Mercuri for pseudo-parallel immersions in Riemannian space forms. As consequences, we prove that any semi-parallel spacelike hypersurface with zero mean curvature in the de Sitter spacetime is totally geodesic. In contrast, there is no semi-parallel spacelike hypersurface with zero mean curvature in the Einstein-de Sitter spacetime. 
\end{abstract}



\maketitle


\section{Introduction}
Simons' type formula is an important tool for studying the class of submanifolds with parallel mean curvature vector field. The initial foundation for this was provided by James Simons (\cite{Si}) in his renowned article ``\textit{Minimal varieties in Riemannian manifolds}'' from which we learn that the Laplacian of the squared norm of the second fundamental form of a minimal hypersurface in the unit Euclidean sphere $\mathbb S^{n+1}(1)$ satisfies
\begin{align}\label{simon}\frac{1}{2}\Delta S=nS-S^2+||\nabla A||^2,\end{align} where $S$ and $A$ denote the squared norm of the second fundamental form and the shape operator of hypersurface, respectively. In this work, he also showed that a closed, $n$-dimensional minimal submanifold $M$ in $\mathbb S^{n+k}(1)$ is either totally geodesic, or $S=\dfrac{n}{2-\frac{1}{k}}$, or at some $p\in M^n, S(p)>\dfrac{n}{2-\frac{1}{k}}$.
Two years later, Chern, Do Carmo and Kobayashi in \cite{Chern} 
showed that 
$\mathbb S^m(\sqrt{\frac{m}{n}})\times\mathbb S^{n-m}(\sqrt{\frac{n-m}{n}})$ of $\mathbb S^{n+1}(1)$ and the Veronese surface of $\mathbb S^4$ are
the only compact minimal submanifolds satisfying $S=\dfrac{n}{2-\frac{1}{k}}$.

It is worth mentioning that the above results about minimal submanifolds of $\mathbb S^{n+k}$ also are related to the original version of the conjecture proposed by Chern in \cite{Conjecture}: ``\textit{Let $M^n\subset\mathbb S^{n+1}(1)$ be a closed, minimally immersed hypersurface with constant scalar
curvature $R_M$. Then for each $n$, the set of all the possible values for $R_M$ (or equivalently $S$) is discrete}''. For more detail about this conjecture we recommend to see \cite{Ge,Sche,Tang}.

Starting in 1969, various generalizations of Simons' formula were developed to investigate the classification of submanifolds with either constant mean curvature or parallel mean curvature vector field when immersed in space forms or product spaces. These generalized formulas came to be known as Simons' type formulas.
In the following paragraphs, we will highlight some significant contributions related to Simons' type formulas.

Nomizu and Smyth in \cite{NS} proved that any constant mean curvature hypersurface immersed in a Riemannian manifold with constant curvature $c$ satisfies 
\begin{align}\label{type}
 \frac{1}{2}\Delta S=cn S-S^2-c(tr A)^2+(tr A)(tr A^3)+||\nabla A||^2,
 \end{align}
 where $tr$ denotes the trace function. As an application, they classified the hypersurfaces with non-negative sectional curvature
 immersed in either Euclidean space $\mathbb R^{n+1}$ or in the unit Euclidean sphere $\mathbb S^{n+1}$. In the context of isometric immersion into product space, we would like to highlight recent works by Fetcu and Rosenberg \cite{DH}, Santos \cite{Fabio1}, and Santos and da Silva \cite{Fabio}. In \cite{DH}, the authors derived a Simons' type formula for submanifolds with parallel mean curvature vector in the product space $\mathbb R\times\mathbb Q^n(c)$, where $\mathbb Q^n(c)$ denotes the Riemannian space form with constant curvature $c$, and they subsequently provided characterizations for complete submanifolds within these spaces.
 Conversely, in \cite{Fabio1} the Cheng-Yau operator was employed to establish a Simons' type formula for surfaces
 in $\mathbb R\times\mathbb Q^2(c), c\in\{-1,1\}$ and as a consequence, they obtained that the complete surfaces with constant extrinsic curvature are isometric to either a cylinder $\mathbb R\times \mathbb H^1$ if $c=-1$, or a slice $\{t_0\}\times \mathbb S^2$ if $c=1$, for some $t_0\in\mathbb R$.
 In \cite{Fabio}, the authors presented a Simons' type formula similar to the one in \cite{DH}, for submanifolds with parallel normalized mean curvature (\textit{pnmc}) in the $\mathbb R\times\mathbb S^n$. They subsequently provided an integral inequality for \textit{pnmc} closed submanifolds of $\mathbb R\times\mathbb S^n$.

It is worth noting that there are numerous applications of Simons' type formula in the context of semi-Riemannian manifolds. For example, Chaves and Souza in \cite{Chaves} established Simons' type inequalities for submanifolds in semi-Riemannian space form, thereby obtaining a Lorentzian version of certain results related to umbilical submanifolds. Additionally, Lima at al. \cite{Gomes1} employed a Simons' type formula for spacelike hypersurfaces in Lorentz spaces to investigate complete linear Weingarten spacelike hypersurfaces with two distinct principal curvatures.

The increasing interest in semi-Riemannian geometry among geometers and physicists, particularly in the context of submanifolds immersed in semi-Riemannian warped products, is indeed remarkable, as seen in recent research \cite{alma,Colares,MM,RI}. This diverse family of spaces includes Robertson-Walker spacetime, which is a Lorentzian warped product having as base an open interval $I\subset\mathbb R$ endowed with the metric $-dt^2$ and as fiber a Riemannian manifold endowed with a metric $g$ of constant curvature. Basic examples of such spacetimes are Lorentz-Minkowski, de Sitter, Einstein-de Sitter, and anti-de Sitter spacetimes. 

  In this paper, we study Simons' type formula for spacelike submanifolds immersed in semi-Riemannian warped products of the form
 $\varepsilon I\times_a\mathbb Q^{n+m}_s(c)$, where $\mathbb Q_s^{n+m}(c)$ denotes an $n+m-$dimensional semi-Riemannian space form of index $s$ with constant curvature $c$, $a:I\subset\mathbb R\rightarrow\mathbb R_+$ is a smooth function and $\varepsilon=\pm 1$ (as described in section \ref{sec2}). It is important to note that the recent work by Ribeiro and de Melo \cite{RI} played a significant role in establishing our Simons' type formula. In the following paragraphs, we will provide a brief overview of the main results of our research.
 
In order to state our results, let $f:M^n\rightarrow\varepsilon I\times_a\mathbb Q_s^{n+m}(c)$ be a spacelike isometric immersion, which means the induced metric by $f$ on $M$ from the ambient space $\varepsilon I\times_a\mathbb Q_s^{n+m}$ is positive definite. 
Let $\{e_1,\ldots,e_{n+m+1}\}$ be an adapted frame on $M$, i.e., $\{e_1,\ldots,e_n\}$ is a local orthonormal frame for $TM$ and $\{e_{n+1},\ldots,e_{n+m+1}\}$ is a local frame for $TM^\perp$ such that $\langle e_\gamma,e_\beta\rangle=\varepsilon_\gamma\delta_{\gamma\beta}, \forall\beta,\gamma\in\{n+1,\ldots,n+m+1\}$ and $\varepsilon_\gamma\in\{-1,1\}$. Let $\frac{\partial}{\partial t}$ be a vector field tangent to the first factor and we denote by $T$ and $\xi$ its tangent and normal parts to $M$, respectively. Let $\alpha$ be the second fundamental form on $M$ and denote $h^{\beta}_{ij}=\langle \alpha(e_i,e_j), e_\beta\rangle, i,j=1,\ldots,n$ (see section  \ref{sec2}). Then, as our first main result, Simons' type formula for spacelike submanifolds in $\varepsilon I\times_a\mathbb Q^{n+m}_s(c)$ can be stated as follows.  
 
\begin{theorem}\label{T1}
    Let $f:M^n\rightarrow\varepsilon I\times_a\mathbb Q_s^{n+m}(c)$ be a spacelike isometric immersion. Then,  
     \begin{align*}
\dfrac{1}{2}\Delta|\alpha|^2=& \sum_\beta\sum_{i,j,k}h^\beta_{ij}h^\beta_{kkji}+|\nabla^\perp\alpha|^2+\sum_\beta\Big\langle \nabla\Big(\dfrac{a''}{a}-\varepsilon b\Big),(nA_\beta-(trA_\beta)I)(T)\Big\rangle\langle\xi,e_\beta\rangle\\ 
 &-\sum_\beta b'\Big\langle (nA_\beta-(trA_\beta)I)(T),T \Big\rangle\langle\xi,e_\beta\rangle+\sum_{\beta}\Big(\dfrac{a''}{a}-\varepsilon b\Big)\Big(ntr(A_\xi A_\beta)-(trA_\xi)(trA_\beta)\Big)\Big)\langle\xi,e_\beta\rangle\\
 &+\sum_\beta\Big(\dfrac{a''}{a}-\varepsilon b\Big)\Big(3\langle A_\beta T,T\rangle trA_\beta-2n||A_\beta T||^2-||T||^2trA_\beta^2\Big)\\
 &+\sum_{\beta}b((trA_\beta)^2-n(trA_\beta^2))+\sum_{\beta,\gamma}\varepsilon_\gamma\Big((trA_\gamma)tr(A_\beta^2A_\gamma)+tr([A_\gamma,A_\beta]^2)-(tr(A_\beta A_\gamma))^2\Big),
 \end{align*}
 where $b=\dfrac{\varepsilon(a')^2-c}{a^2}$ and $A_\beta$ is the shape operator of $M$ with respect to $e_\beta\in TM^\perp$.
    \end{theorem}
  
\begin{remark} {~{\upshape{
\begin{itemize}
    \item [1.] Theorem \ref{T1} generalizes the Simons' type formulas proved in \cite{DH, NS}. 
    \item [2.] When $\varepsilon I\times_a\mathbb Q_s^{n}(c)$ has constant curvature, it follows from Theorem \ref{T1}  that any compact spacelike hypersurface in $\varepsilon I\times_a\mathbb Q_s^{n}(c)$ with parallel mean curvature vector field and non-negative sectional curvature is isoparametric, i.e., the eigenvalues of the shape operator are constants (see Corollary \ref{co3}). This result generalizes the Riemannian case as obtained by Nomizu and Smyth in \cite{NS}.
\end{itemize}}}}  
\end{remark}

As a further application of Theorem \ref{T1}, we will extend the result on minimal pseudo-parallel immersion in space forms
 originally obtained by Asperti, Lobos and Mercuri in \cite{AS}.
 
   
\begin{theorem}\label{c1}
   Let $f:M^n\rightarrow \varepsilon I\times_a\mathbb Q_s^{n+m}(c)$ be a $\psi$-pseudo-parallel spacelike immersion and let $\Vec{H}$ be the mean curvature vector field of $f$ such that $\Vec{H}(p)=0, p\in M$. 
   \begin{enumerate}
       \item For $s=0$ and $\varepsilon=1$, if $\frac{a''}{a}-\frac{(a')^2}{a^2}+\frac{c}{a^2}\geq0$ and
   \begin{align*}
   \psi(p)\geq -\frac{1}{n}\Big(\Big(\frac{a''}{a}-\frac{(a')^2}{a^2}+\frac{c}{a^2}\Big)||T||^2+n\Big(\frac{(a')^2- c}{a^2}\Big)\Big),\end{align*} 
   then $p$ is a geodesic point.
   \item For $s=0$ and $\varepsilon=-1$ (Robertson-Walker spacetime) or $0<s=m+\frac{1+\varepsilon}{2}$, $\varepsilon\in\{-1,1\}$, if $\frac{a''}{a}-\frac{(a')^2}{a^2}+\frac{\varepsilon c}{a^2}\leq0$ and
   \begin{align*}
   \psi(p)\leq -\frac{1}{n}\Big(\Big(\frac{a''}{a}-\frac{(a')^2}{a^2}+\frac{\varepsilon c}{a^2}\Big)||T||^2+n\Big(\frac{\varepsilon(a')^2- c}{a^2}\Big)\Big),\end{align*} then $p$ is a geodesic point.
   \end{enumerate} 
\end{theorem}
\begin{remark} {~{\upshape{
\begin{itemize}
\item[1.]  Pseudo-parallel immersions were initially introduced in \cite{AS} as natural extension of semi-parallel submanifolds and as the extrinsic analogue of pseudo-symmetric manifolds introduced by Deszcz in \cite{D}.
    \item [2.]  The assumptions in Theorem \ref{c1} are indeed necessary, as discussed in section \ref{parallel}. 
    \item [3.] When the ambient manifold is the de Sitter spacetime, it can be deduced from Theorem \ref{c1} that any semi-parallel spacelike hypersurface ($\psi=0$) with $\Vec{H}(p)=0$ are totally geodesic at $p$ (see Corollary \ref{deSitter}). From a global point of view, T. Ishihara in \cite{TI} proved that any complete maximal submanifolds in semi-Riemannian manifold $L^{n+s}_s(c)$ of constant curvature $c\geq0$ is totally geodesic. Years later, Alias and Romero in \cite{AL} established that the only $n$-dimensional complete maximal submanifolds in $\mathbb S^{n+s}_s$ are totally geodesic ones.
    
    
    \item [4.] Another consequence of Theorem \ref{c1} refers to the nonexistence of semi-parallel spacelike hypersurface with zero mean curvature in the Einstein-de Sitter spacetime. This outcome is a partial extension  of a result established by Aledo et. al \cite{Aledo}, which concerns the nonexistence of totally geodesic spacelike hypersurfaces in the Einstein-de Sitter spacetime.
    \item [5.] When $f$ has codimension $2$, we obtain that normal bundle of $f$ is flat if either the assumptions of Theorem \ref{c1} hold or $\Vec H(p)\neq0$ for some $p\in M$, where $\Vec{H}$ is the mean curvature vector field of $f$.  This generalizes a similar result for pseudo-parallel immersions of codimension $2$ in space forms by Asperti et al. \cite{AS}. 
\end{itemize}}}}   
\end{remark}

This paper is organized as follows: In section \ref{sec2} we present the basic concepts of theory of submanifolds in semi-Riemannian warped product. In section \ref{proof} we prove  Theorem \ref{T1} and also present its direct applications. In section \ref{parallel} we discuss about pseudo-parallel immersions and we demonstrate Theorem \ref{c1} and its consequences. 

\section{Preliminaries}\label{sec2}
Let $(\mathbb Q_s^{n+m}(c),g)$ be an $n+m-$dimensional semi-Riemannian space
form of index $s$, constant curvature $c$ and metric $g$. Namely, $\mathbb Q_s^{n+m}(c)$ may be considered up to isometries, as the semi-Riemannian hyperbolic space $\mathbb H_s^{n+m}(c)$ if $c<0$, the semi-Riemannian Euclidean space $\mathbb R_s^{n+m}$ if $c=0$ and the semi-Riemannian sphere $\mathbb S_s^{n+m}(c)$ if $c>0$. For $s=0$, we denote $\mathbb Q^{n+m}_0(c)=\mathbb Q^{n+m}(c)$ which is a Riemannian space form. For $s=1$, the manifolds $\mathbb R^{n+m}_1, \mathbb S^{n+m}_1(c)$ and $\mathbb H^{n+m}_1(c)$ are called  Lorentz-Minkowski, de Sitter and anti-de Sitter spacetime in the general relativity.

We consider $\varepsilon I\times_a\mathbb Q_s^{n+m}(c)$ the semi-Riemannian warped product equipped with a warped product metric \[\langle\cdot,\cdot\rangle=\varepsilon \pi^\ast_I(dt^2)+a(\pi_I)^2\pi^\ast_{\mathbb Q}g,\] where $\pi_I, \pi_{\mathbb Q}$ denote the projections onto $I$ and $\mathbb Q$, respectively, $a:I\subset\mathbb R\rightarrow\mathbb R_+$ is a smooth function and $\varepsilon=\pm 1$. 
In \cite{RI}, we can find that the curvature tensor $\overline{R}$ of $\varepsilon I\times_a\mathbb Q_s^{n+m}(c)$ is given by
\begin{align}\label{Rbarra}\nonumber
    \overline R(X,Y,Z,W)=&\Big(\varepsilon \frac{(a')^2}{a^2}-\frac{c}{a^2}\Big)\Big(\langle X,Z\rangle\langle Y,W\rangle-\langle Y,Z\rangle\langle X,W\rangle\Big)\\ \nonumber
    &+\Big(\frac{a''}{a}-\frac{(a')^2}{a^2}+\varepsilon\frac{c}{a^2}\Big)\Big(\langle X,Z\rangle\langle Y,\partial t\rangle\langle W,\partial t\rangle-\langle Y,Z\rangle\langle X,\partial t\rangle\langle W,\partial t\rangle\\
   & -\langle X,W\rangle\langle Y,\partial t\rangle\langle Z,\partial t\rangle
    +\langle Y,W\rangle\langle X,\partial t\rangle\langle Z,\partial t\rangle\Big), 
\end{align}
for each $X, Y, Z, W$ on $\varepsilon I\times_a\mathbb Q_s^{n+m}(c).$

One check easily from \eqref{Rbarra} that $\varepsilon I\times_a\mathbb Q_s^{n+m}(c)$ has constant curvature $\kappa$ if and only if 
\begin{align}\label{constant}
\frac{a''}{a}=\frac{(a')^2}{a^2}-\frac{ \varepsilon c}{a^2}=-\varepsilon\kappa.    
\end{align}

Let $f:M^n\rightarrow \varepsilon I\times_a\mathbb Q_s^{n+m}(c)$ be an isometric immersion of a semi-Riemannian manifold $M^n$ in $\varepsilon I\times_a\mathbb Q_s^{n+m}(c)$. In the sequence, every geometric object is related to the immersion $f$.  When there is no confusion regarding which immersion we are considering, we will refer to the submanifold $M$ instead of the immersion $f$. 
 
For any $X, Y\in TM$, the Gauss and Weingarten formulas are given by 
\[
\overline\nabla_XY=\nabla_XY+\alpha(X,Y)\quad\mbox{and}\quad \overline\nabla_X\eta=-A_\eta X +\nabla^\perp_X\eta,
\]
where $\overline\nabla$ and $\nabla$ denote the Levi-Civita connection on $\varepsilon I\times_a\mathbb Q_s^{n+m}(c)$ and $M$, respectively, $\alpha\in Hom^2(TM,TM;TM^\perp)$ is the second fundamental form of $M$ and $A_\eta\in Hom(TM;TM)$ is the shape operator associated to normal direction $\eta\in TM^\perp$ such that $\langle A_\eta(X), Y\rangle=\langle \alpha(X,Y),\eta\rangle$. 
Moreover, the mean curvature vector field of $M$ is defined by $\Vec{H}=\dfrac{1}{n}tr(\alpha)$, where $tr$ denotes the trace function.  

We denote by $\pi_M:M\rightarrow I$ the restriction to $M$ of the natural projection $\pi_I$. So, $\partial t=T+\xi,$ where $T=\varepsilon \mbox{grad}(\pi_M)$ and $\xi\in TM^\perp$, in particular, we have that $\varepsilon=\langle\partial t,\partial t\rangle=\langle T,T\rangle+\langle\xi,\xi\rangle$.
In \cite{RI}, it was proved the following identities:
\begin{align}
    \overline\nabla_X\partial t=&\frac{a'}{a}(X-\varepsilon\langle X,\partial t\rangle\partial t)\\ \label{id1}
    \nabla_XT=&\frac{a'}{a}(X-\varepsilon\langle X,T\rangle T)+A_\xi X\\ \label{id2}
    \nabla^\perp_X\xi=&-\varepsilon\frac{a'}{a}\langle X,T\rangle\xi-\alpha(X,T).
\end{align}

We now denote $R$ and $R^\perp$ the curvature tensors of the tangent and normal bundles $TM$ and $TM^\perp$. Then, for each $X, Y, Z, W\in TM$ and $\eta\in TM^\perp$, the fundamental equations are given by
\begin{enumerate}
\item Gauss equation: \begin{align*}
        R(X,Y,Z,W)=&\Big(\frac{\varepsilon(a')^2 -c}{a^2}\Big)\Big(\langle X,Z\rangle\langle Y,W\rangle-\langle Y,Z\rangle\langle X,W\rangle\Big)\\
        &+\Big(\frac{a''}{a}+\frac{\varepsilon c-(a')^2}{a^2}\Big)\Big(\langle X,Z\rangle\langle Y,T\rangle\langle W,T\rangle-\langle Y,Z\rangle\langle X,T\rangle\langle W,T\rangle\\&+\langle Y,W\rangle\langle X,T\rangle\langle Z,T\rangle-\langle X,W\rangle\langle Y,T\rangle\langle Z,T\rangle\Big)\\
        &-\langle \alpha(X,Z),\alpha(Y,W)\rangle+\langle\alpha(X,W),\alpha(Y,Z)\rangle,
    \end{align*}
\item Codazzi equation:
\begin{align*}
   (\nabla^\perp_Y\alpha)(X,Z,\eta)&=(\nabla^\perp_X\alpha)(Y,Z,\eta)+\Big(\frac{a''}{a}+\frac{\varepsilon c-(a')^2}{a^2}\Big)\Big(\langle T,X\rangle\langle Y,Z\rangle-\langle T,Y\rangle\langle X,Z\rangle\Big)\langle\xi,\eta\rangle,
\end{align*}
\item Ricci equation:
\begin{align*}
    R^\perp(X,Y)\eta=\alpha(A_\eta Y,X)-\alpha(A_\eta X,Y).
\end{align*}
\end{enumerate}

Throughout this paper, we will consider $f:M\rightarrow\varepsilon I\times_a\mathbb Q_s^{n+m}(c)$ as a spacelike isometric immersion. 
Additionally, we refer to $f$ as \textit{extremal} (maximal or minimal) if its mean curvature vector field vanishes identically. 

\section{Proof of Theorem \ref{T1}}\label{proof}

Let $p$ be an arbitrary point in $M$ and we consider around it an adapted frame $\{e_1,\ldots,e_{n+m+1}\}$ on $M$, i.e., $\{e_1,\ldots,e_n\}$ is a local orthonormal frame for $TM$ such that $(\nabla e_i)(p)=0$, and $\{e_{n+1},\ldots,e_{n+m+1}\}$ is a local frame for $TM^\perp$ such that $\langle e_\gamma,e_\beta\rangle=\varepsilon_\gamma\delta_{\gamma\beta}, \forall\beta,\gamma\in\{n+1,\ldots,n+m+1\}$ and $\varepsilon_{\gamma}\in\{-1,1\}$. In what follows, we shall use the following convention on the ranges of indices:
\begin{align*}
    1\leq i,j,k,l\leq n\quad\mbox{and}\quad n+1\leq\beta,\gamma\leq n+m+1.
\end{align*}

We denote by $\nabla_je_i=\nabla_{e_j}e_i$,
$
h^\beta_{ij}=\langle\alpha(e_i,e_j),e_\beta\rangle$, $h^\beta_{ijk}=\langle(\nabla^\perp_k\alpha)(e_i,e_j), e_\beta\rangle$ and $h^\beta_{ijkl}=\langle(\nabla^\perp_l\nabla_k^\perp\alpha)(e_i,e_j),e_\beta\rangle$. 
Thus, we define the squared norm of the second fundamental form by $|\alpha|^2=\displaystyle\sum_{\beta}\sum_{i,j}(\varepsilon_\beta h_{ij}^\beta)^2=\displaystyle\sum_{\beta}\sum_{i,j}(h_{ij}^\beta)^2$, as well as  $|\nabla^\perp\alpha|^2=\displaystyle\sum_\beta\sum_{i,j,k}(h^\beta_{ijk})^2$. For simplicity, we shall denote by $b=\frac{\varepsilon(a')^2-c}{a^2}$ and $B=\frac{a''}{a}-\varepsilon b$.

This way the Laplacian of the squared norm of the second fundamental form of $M$ is given by
\begin{align}\label{soma}\nonumber
    \dfrac{1}{2}\Delta|\alpha|^2&=\sum_\beta\sum_{i,j,k}\Big((h^\beta_{ijk})^2+h^\beta_{ij}h^\beta_{ijkk}\Big)\\
    &=|\nabla^\perp\alpha|^2+\sum_\beta\sum_{i,j,k}h^\beta_{ij}h^\beta_{ijkk}.
 \end{align}
We then compute the second term of \eqref{soma} as follows.
\begin{align}\label{nabla}
  (\nabla^\perp_k\nabla^\perp_k\alpha)(e_i,e_j)=&\nabla^\perp_k\Big((\nabla^\perp_k\alpha)(e_i,e_j)\Big)\,\,\mbox{at}\,\, p.  
\end{align}
From Codazzi equation we have that
\begin{align*}
    \nabla^\perp_k\Big((\nabla^\perp_k\alpha)(e_i,e_j)\Big)=&\nabla^\perp_k\Big((\nabla^\perp_i\alpha)(e_k,e_j)\Big)+\nabla^\perp_k\Big(B\Big(\langle T,e_i\rangle\delta_{kj}-\langle T,e_k\rangle\delta_{ij}\Big)\xi\Big).
\end{align*}
Using \eqref{id1} and \eqref{id2} in the last equality, we get 
\begin{align*}
\nabla^\perp_k\Big((\nabla^\perp_k\alpha)(e_i,e_j)\Big)=&\nabla^\perp_k\Big((\nabla^\perp_i\alpha)(e_k,e_j)\Big)+e_k(B)(\langle T,e_i\rangle\delta_{kj}-\langle T,e_k\rangle\delta_{ij})\xi\\
&+\frac{a'}{a}B\Big((\delta_{ki}\delta_{kj}-\delta_{kk}\delta_{ij})+2\varepsilon\Big(\langle T,e_k\rangle^2\delta_{ij}-\langle T,e_i\rangle\langle T,e_k\rangle\delta_{kj}\Big)\Big)\xi\\
    &+B\Big(\langle A_{\xi}e_k,e_i\rangle\delta_{kj}-\langle A_\xi e_k,e_k\rangle\delta_{ij}\Big)\xi-B\Big(\langle T,e_i\rangle\delta_{kj}-\langle T,e_k\rangle\delta_{ij}\Big)\alpha(e_k,T)\,\,\mbox{at}\,\, p.
\end{align*}
Since,
\begin{align*}
\nabla^\perp_k\Big((\nabla^\perp_i\alpha)(e_k,e_j)\Big)=&\nabla^\perp_k\Big((\nabla^\perp_i\alpha)(e_j,e_k)\Big)\\
=&\nabla^\perp_i\nabla_k^\perp\alpha(e_j,e_k)
+R^\perp(e_k,e_i)\alpha(e_j,e_k)+\nabla_{[e_k,e_i]}\alpha(e_j,e_k)\\
&-(\nabla^\perp_k\alpha)(\nabla_ie_j,e_k)-\alpha(\nabla_k\nabla_ie_j,e_k)-\alpha(\nabla_ie_j,\nabla_ke_k)\\
&-(\nabla^\perp_k\alpha)(e_j,\nabla_ie_k)-\alpha(\nabla_ke_j,\nabla_ie_k)-\alpha(e_j,\nabla_k\nabla_ie_k),
\end{align*}
it follows that, 
\begin{align*}
\nabla^\perp_k\Big((\nabla^\perp_i\alpha)(e_k,e_j)\Big)&=\nabla^\perp_i\nabla_k^\perp\alpha(e_j,e_k)
+R^\perp(e_k,e_i)\alpha(e_j,e_k)-\alpha(\nabla_k\nabla_ie_j,e_k)-
\alpha(e_j,\nabla_k\nabla_ie_k)\,\,\mbox{at}\,\, p.
\end{align*}
Now, we see that
\begin{align*}
\nabla^\perp_i\nabla_k^\perp\alpha(e_j,e_k)=& \nabla^\perp_i\Big((\nabla_k^\perp\alpha)(e_j,e_k)\Big)+(\nabla_i^\perp\alpha)(\nabla_ke_j,e_k)+\alpha(\nabla_i\nabla_ke_j,e_k)\\
&+\alpha(\nabla_ke_j,e_k)+\nabla^\perp_i\alpha(e_j,\nabla_ke_k)+\alpha(\nabla_ie_j,\nabla_ke_k)+\alpha(e_j,\nabla_i\nabla_ke_k).
\end{align*}
So,
\begin{align*}
  \nabla^\perp_i\nabla_k^\perp\alpha(e_j,e_k)=& \nabla^\perp_i\Big((\nabla_k^\perp\alpha)(e_j,e_k)\Big)+\alpha(\nabla_i\nabla_ke_j,e_k)+\alpha(e_j,\nabla_i\nabla_ke_k)\,\,\mbox{at}\,\, p. 
\end{align*}
Furthermore, from Codazzi equation together with \eqref{id1} and  \eqref{id2} we have
\begin{align*}
\nabla^\perp_i\nabla_k^\perp\alpha(e_j,e_k)=&\Big(\nabla_i^\perp\nabla_j^\perp\alpha\Big)(e_k,e_k)+e_i(B)(\langle T,e_j\rangle\delta_{kk}-\langle T,e_k\rangle\delta_{jk})\xi\\
 &+\frac{a'}{a}B\Big(\delta_{ij}\delta_{kk}-\delta_{ik}\delta_{jk}+2\varepsilon(\langle T,e_k\rangle\langle T,e_i\rangle\delta_{jk}-\langle T,e_i\rangle\langle T,e_j\rangle\delta_{kk})\Big)\xi\\
 &+B\Big(\langle A_\xi e_i,e_j\rangle\delta_{kk}-\langle A_\xi e_i,e_k\rangle\delta_{jk}\Big)\xi-B\Big(\langle T,e_j\rangle\delta_{kk}-\langle T,e_k\rangle\delta_{jk}\Big)\alpha(e_i,T)\\
&+\alpha(\nabla_i\nabla_ke_j,e_k)+\alpha(e_j,\nabla_i\nabla_ke_k)\,\,\mbox{at}\,\, p. 
\end{align*}
Then,
\begin{align*}
 \nabla^\perp_k\Big((\nabla^\perp_i\alpha)(e_k,e_j)\Big)&= \Big(\nabla_i^\perp\nabla_j^\perp\alpha\Big)(e_k,e_k)+e_i(B)(\langle T,e_j\rangle\delta_{kk}-\langle T,e_k\rangle\delta_{jk})\xi\\
 &+\frac{a'}{a}B\Big(\delta_{ij}\delta_{kk}-\delta_{ik}\delta_{jk}+2\varepsilon(\langle T,e_k\rangle\langle T,e_i\rangle\delta_{jk}-\langle T,e_i\rangle\langle T,e_j\rangle\delta_{kk})\Big)\xi\\
 &+B\Big(\langle A_\xi e_i,e_j\rangle\delta_{kk}-\langle A_\xi e_i,e_k\rangle\delta_{jk}\Big)\xi-B\Big(\langle T,e_j\rangle\delta_{kk}-\langle T,e_k\rangle\delta_{jk}\Big)\alpha(e_i,T)\\ 
&+\alpha(R(e_i,e_k)e_j,e_k)+\alpha(e_j,R(e_i,e_k)e_k)+R^\perp(e_k,e_i)\alpha(e_j,e_k)\,\,\mbox{at}\,\, p.   
\end{align*}
 From the last identity we have that \eqref{nabla} becomes
\begin{align*}
(\nabla^\perp_k\nabla^\perp_k\alpha)(e_i,e_j)=& \Big(\nabla_i^\perp\nabla_j^\perp\alpha\Big)(e_k,e_k)+e_i(B)(\langle T,e_j\rangle\delta_{kk}-\langle T,e_k\rangle\delta_{jk})\xi\\
&+e_k(B)(\langle T,e_i\rangle\delta_{kj}-\langle T,e_k\rangle\delta_{ij})\xi+2\varepsilon B\frac{a'}{a}\Big(\langle T,e_k\rangle^2\delta_{ij}-\langle T,e_i\rangle\langle T,e_j\rangle\delta_{kk}\Big)\xi\\
&+B\Big(\langle A_\xi e_i,e_j\rangle\delta_{kk}-\langle A_\xi e_k,e_k\rangle\delta_{ij}\Big)\xi-B(\langle T,e_i\rangle\delta_{kj}-\langle T,e_k\rangle\delta_{ij}\Big)\alpha(e_k,T)\\
&-B\Big(\langle T,e_j\rangle\delta_{kk}-\langle T,e_k\rangle\delta_{jk})\alpha(e_i,T)\Big) +\alpha(R(e_i,e_k)e_j,e_k)\\
&+\alpha(e_j,R(e_i,e_k)e_k)+R^\perp(e_k,e_i)\alpha(e_j,e_k)\,\,\mbox{at}\,\, p. 
\end{align*}
Therefore,
\begin{align*}
\sum_\beta\sum_{i,j,k}h_{ij}^\beta h^\beta_{ijkk}=&\sum_\beta\sum_{i,j,k}h^\beta_{ij}h^\beta_{kkji}\\
&+\sum_\beta\sum_{i,j,k}h^\beta_{ij}e_i(B)\Big(\langle T,e_j\rangle\delta_{kk}-\langle T,e_k\rangle\delta_{jk}\Big)\langle\xi,e_\beta\rangle\\
&+\sum_\beta\sum_{i,j,k}h^\beta_{ij}e_k(B)\Big(\langle T,e_i\rangle\delta_{kj}-\langle T,e_k\rangle\delta_{ij}\Big)\langle\xi,e_\beta\rangle\\
&+2\varepsilon \frac{a'}{a}B\sum_\beta\sum_{i,j,k}h^\beta_{ij}\Big(\langle T,e_k\rangle^2\delta_{ij}-\langle T,e_i\rangle\langle T,e_j\rangle\delta_{kk}\Big)\langle\xi,e_\beta\rangle\\
&+B\sum_\beta\sum_{i,j,k}h^\beta_{ij}\Big((\langle A_\xi e_i,e_j\rangle\delta_{kk}-\langle A_\xi e_k,e_k\rangle\delta_{ij}\Big)\langle\xi,e_\beta\rangle\\
&-B\sum_\beta\sum_{i,j,k}h^\beta_{ij}\Big(\langle T,e_i\rangle\delta_{kj}-\langle T,e_k\rangle\delta_{ij}\Big)\langle\alpha(e_k,T),e_\beta\rangle\\
&-B\sum_\beta\sum_{i,j,k}h^\beta_{ij}\Big(\langle T,e_j\rangle\delta_{kk}-\langle T,e_k\rangle\delta_{jk}\Big)\langle\alpha(e_i,T),e_\beta\rangle\\
&+\sum_\beta\sum_{i,j,k}h^\beta_{ij}\langle\alpha(R(e_i,e_k)e_j,e_k)+\alpha(e_j,R(e_i,e_k)e_k),e_\beta\rangle\\
&+\sum_\beta\sum_{i,j,k}h^\beta_{ij}\langle R^\perp(e_k,e_i)\alpha(e_j,e_k),e_\beta\rangle\,\,\mbox{at}\,\, p. 
\end{align*}
It is easy see that
\begin{align*}
 \sum_\beta\sum_{i,j,k}h_{ij}^\beta e_i(B)\Big(\langle T,e_j\rangle\delta_{kk}-\langle T,e_k\rangle\delta_{jk}\Big)=&\sum_\beta(n-1)\langle\alpha(\nabla B,T),e_\beta\rangle,
\end{align*}
\begin{align*}\label{eq2}
 \sum_\beta\sum_{i,j,k}h_{ij}^\beta e_k(B)\Big(\langle T,e_i\rangle\delta_{kj}-\langle T,e_k\rangle\delta_{ij}\Big)=&\sum_\beta(n\langle\alpha(\nabla B,T),e_\beta\rangle-\langle T,\nabla B\rangle tr(A_\beta)),
\end{align*}
\begin{align*}
    \sum_\beta\sum_{i,j,k}h_{ij}^\beta \Big(\langle T,e_k\rangle^2\delta_{ij}-\langle T,e_i\rangle\langle T,e_j\rangle\delta_{kk}\Big)=&\sum_\beta\Big(||T||^2tr(A_\beta)-n\langle\alpha(T,T),e_\beta\rangle\Big),
\end{align*}
\begin{align*}
    \sum_\beta\sum_{i,j,k}h_{ij}^\beta\Big(\langle A_\xi e_i,e_j\rangle\delta_{kk}-\langle A_\xi e_k,e_k\rangle\delta_{ij}\Big)=&\sum_\beta\Big(ntr(A_\xi A_\beta)-(tr(A_\xi))(tr(A_\beta))\Big),
\end{align*}
\begin{align*}
   \sum_\beta\sum_{i,j,k}h_{ij}^\beta\Big(\langle T,e_i\rangle\delta_{kj}-\langle T,e_k\rangle\delta_{ij}\Big)\langle\alpha(e_k,T),e_\beta\rangle=&\sum_\beta\Big(||A_\beta T||^2-\langle A_\beta T,T\rangle trA_\beta\Big)
\end{align*}
and
\begin{align*}
    \sum_\beta\sum_{i,j,k}h_{ij}^\beta\Big(\langle T,e_j\rangle\delta_{kk}-\langle T,e_k\rangle\delta_{jk}\Big)\langle\alpha(e_i,T),e_\beta\rangle=&\sum_\beta(n-1)||A_\beta T||^2.
\end{align*}
Now, using the Gauss equation we obtain that 
\begin{align*}
\langle\alpha(R(e_i,e_k)e_j,e_k),e_\beta\rangle=&b\Big(\delta_{ij}\langle e_k,A_\beta e_k\rangle-\delta_{kj}\langle e_i,A_\beta e_k\rangle\Big)\\
&+B\Big(\delta_{ij}\langle e_k,T\rangle\langle A_\beta e_k,T\rangle-\delta_{kj}\langle e_i,T\rangle\langle A_\beta e_k,T\rangle\Big)\\
&-B\Big(\langle e_i,A_\beta e_k\rangle\langle e_k,T\rangle\langle e_j,T\rangle-\langle e_k,A_\beta e_k\rangle\langle e_i,T\rangle\langle e_j,T\rangle\Big)\\
&-\langle\alpha(e_i,e_j),\alpha(e_k,A_\beta e_k)\rangle+\langle\alpha(e_i,A_\beta e_k),\alpha(e_k,e_j)\rangle,
\end{align*}
\begin{align*}
\langle\alpha(R(e_i,e_k)e_k,e_j),e_\beta\rangle=&b\Big(\delta_{ik}\langle e_k,A_\beta e_j\rangle-\delta_{kk}\langle e_i,A_\beta e_j\rangle\Big)\\
&+B\Big(\delta_{ik}\langle e_k,T\rangle\langle A_\beta e_j,T\rangle-\delta_{kk}\langle e_i,T\rangle\langle A_\beta e_j,T\rangle\Big)\\
&-B\Big(\langle e_i,A_\beta e_j\rangle\langle e_k,T\rangle\langle e_k,T\rangle-\langle e_k,A_\beta e_j\rangle\langle e_i,T\rangle\langle e_k,T\rangle\Big)\\
&-\langle\alpha(e_i,e_k),\alpha(e_k,A_\beta e_j)\rangle+\langle\alpha(e_i,A_\beta e_j),\alpha(e_k,e_k)\rangle,
\end{align*}
and from Ricci equation we have 
\begin{align*}
    \langle R^\perp(e_k,e_i)\alpha(e_j,e_k),e_\beta\rangle=&\langle\alpha(A_{\alpha(e_j,e_k)}e_i,e_k)-\alpha(A_{\alpha(e_j,e_k)}e_k,e_i),e_\beta\rangle\\
    =&\sum_\gamma\varepsilon_\gamma\langle A_\gamma e_j,e_k\rangle\Big(\langle\alpha(A_\gamma e_i,e_k),e_\beta\rangle-\langle\alpha(A_\gamma e_k,e_i),e_\beta\rangle\Big)\\
    =&\sum_\gamma\varepsilon_\gamma\langle A_\gamma e_j,e_k\rangle\Big(\langle A_\beta(A_\gamma e_i),e_k\rangle-\langle A_\beta(A_\gamma e_k),e_i\rangle\Big)\\
    =&\sum_{\gamma}\varepsilon_\gamma\langle A_\gamma e_j,e_k\rangle\langle [A_\beta,A_\gamma]e_i.e_k\rangle.
\end{align*}
Then,
\begin{align*} \sum_\beta\sum_{i,j,k}h_{ij}^\beta\langle\alpha(R(e_i,e_k)e_j,e_k),e_\beta\rangle=&\sum_\beta b((trA_\beta)^2-trA_\beta^2)\\
&+\sum_\beta 2B(\langle A_\beta T,T\rangle trA_\beta -||A_\beta T||^2)\Big)\\
&+\sum_{\beta,\gamma}\varepsilon_\gamma\Big(-(trA_\gamma A_\beta)^2+\sum_i\langle A_\beta A_\gamma e_i,A_\gamma A_\beta e_i\rangle\Big),
&\\
\sum_\beta\sum_{i,j,k}h_{ij}^\beta\langle\alpha(R(e_i,e_k)e_k,e_j),e_\beta\rangle&=\sum_\beta\Big( (b(1-n)-B||T||^2)trA_\beta^2+B(2-n)||A_\beta T||^2\Big)\\
&+\sum_{\beta,\gamma}\varepsilon_\gamma\Big((trA_\gamma)(trA_\beta^2A_\gamma)-\sum_i\langle A_\beta A_\gamma e_i,A_\beta A_\gamma e_i\rangle\Big)
\end{align*}
and
\begin{align*}
 \sum_\beta\sum_{i,j,k}h_{ij}^\beta\langle R^\perp(e_k,e_i)\alpha(e_j,e_k),e_\beta\rangle=&\sum_{\beta,\gamma}\displaystyle\sum_i\varepsilon_\gamma\langle [A_\beta,A_\gamma] e_i,A_\gamma A_\beta e_i\rangle.  
\end{align*}
Combining the above identities and noting that $b'=2B\varepsilon\frac{a'}{a}$  we can conclude that
 \begin{align*}
 \sum_\beta\sum_{i,j,k}h_{ij}^\beta h^\beta_{ijkk}=&\sum_\beta\sum_{i,j,k}h^\beta_{ij}h^\beta_{kkji}+\sum_\beta\Big(n\langle A_\beta(\nabla B),T\rangle-\langle\nabla B,T\rangle trA_\beta\Big)\langle\xi,e_\beta\rangle\\ 
 &+\sum_\beta b'\Big(||T||^2trA_\beta-n\langle A_\beta T,T\rangle\Big)\langle\xi,e_\beta\rangle\\
 &+\sum_\beta B\Big(ntr(A_\xi A_\beta)-(trA_\xi)(trA_\beta)\Big)\langle\xi,e_\beta\rangle\\
 &+\sum_\beta B\Big(3\langle A_\beta T,T\rangle trA_\beta-2n||A_\beta T||^2-||T||^2trA_\beta^2\Big)\\
 &+\sum_\beta b\Big((trA_\beta)^2-ntrA_\beta^2\Big)\\
&+\sum_{\beta,\gamma}\varepsilon_\gamma\Big((trA_\gamma)tr(A_\beta^2A_\gamma)+tr([A_\gamma,A_\beta]^2)-(tr(A_\beta A_\gamma))^2\Big)
 \end{align*}
which complete the proof of Theorem.
\subsection{Immediate applications of Theorem \ref{T1}}
Herein, we discuss some immediate consequences of Theorem \ref{T1}, starting with the case where $M$ has codimension one.
\begin{corollary}\label{C1}
   Let $f:M^n\rightarrow\varepsilon I\times_a\mathbb Q^n_s(c)$ be a spacelike isometric immersion with shape operator $A$ and second fundamental form $\alpha$. Then,  
    \begin{align*}
\dfrac{1}{2}\Delta|\alpha|^2=& \sum_{i,j,k}h_{ij}h_{kkji}+|\nabla^\perp\alpha|^2+n\Big\langle \nabla\Big(\dfrac{a''}{a}-\varepsilon b\Big),(A-HI_n)(T)\Big\rangle\langle\xi,\nu\rangle\\ 
 &-nb'\Big\langle T, (A-HI_n)(T)\Big\rangle\langle\xi,\nu\rangle+n\Big(\varepsilon\dfrac{a''}{a}- b\Big)\Big(S-nH^2\Big)\langle\xi,\nu\rangle^2\\
 &+\Big(\dfrac{a''}{a}-\varepsilon b\Big)\Big(3n\langle A T,T\rangle H -2n||A T||^2-||T||^2S\Big)\\
 &+ b\Big(n^2H^2-nS\Big)+\delta\Big(nH(trA^3)-S^2\Big),
 \end{align*}
 where $\xi=\partial t-T$, $\nu$ is a local unit normal vector field of $M$ in $\varepsilon I\times_a\mathbb Q^n_s(c)$, i.e., $\langle\nu,\nu\rangle=\delta\in\{-1,1\}$, $I_n$ is identity map of $TM$, $H$ is the mean curvature of $M$ and $S=tr A^2$.

 In particular, if $f$ is extremal, then
 \begin{align*}
\dfrac{1}{2}\Delta|\alpha|^2=& |\nabla^\perp\alpha|^2+n\Big\langle \nabla\Big(\dfrac{a''}{a}-\varepsilon b\Big),AT\Big\rangle\langle\xi,\nu\rangle-nb'\Big\langle AT,T\Big\rangle\langle\xi,\nu\rangle+n\Big(\varepsilon\dfrac{a''}{a}- b\Big)\langle\xi,\nu\rangle^2 S\\
 &-\Big(\dfrac{a''}{a}-\varepsilon b\Big)\Big( 2n||A T||^2+||T||^2S\Big)-nbS
 -\delta S^2.
 \end{align*} 
\end{corollary}

The following result is the Riemannian version of Theorem \ref{T1} for hypersurfaces in space forms.

  \begin{corollary}[Nomizu and Smyth \cite{NS}]\label{NS}
 Let $f:M^n\rightarrow\mathbb Q^{n+1}(c)$ be an isometric immersion. 
 Then,
  \begin{align*}
    \dfrac{1}{2}\Delta|\alpha|^2=&|\nabla^\perp\alpha|^2+\sum_{i,j,k}h_{ij}h_{kkji}+cn(S-nH^2)-S^2+nH(tr A^3).
\end{align*}
In particular, if $H$ is constant, then
 \begin{align*}
    \dfrac{1}{2}\Delta|\alpha|^2=&|\nabla^\perp\alpha|^2+cn(S-nH^2)-S^2+nH(tr A^3).
\end{align*} 
 \end{corollary}
\begin{proof}
We note that $\mathbb Q^{n+1}(c)$ can be described as warped product space $(0,+\infty)\times_{t}\mathbb S^n(1)$ if $c=0$ and $(0,r)\times_{a_c}\mathbb S^n\Big(\frac{1}{|c|}\Big)$ with
\begin{align*}
a_c(t)=\left\{\begin{array}{cc}
    \dfrac{\sin(\sqrt{c}t)}{c},&\mbox{if}\quad c>0\quad\mbox{and}\quad r=\dfrac{\pi}{\sqrt{c}}, \\
      \dfrac{\sinh({\sqrt{|c|}}t)}{|c|},&\mbox{if}\quad c<0\quad\mbox{and}\quad r=+\infty 
\end{array}\right.,\end{align*} Then,
    $b=-c$,
and the proof follows from Corollary \ref{C1}.
\end{proof}

We draw attention that if $\varepsilon I\times_a\mathbb Q_s^n(c), n>1,$ has constant curvature, including the Lorentz-Minkowski space $\mathbb R^{n+1}_1$, the de Sitter space $\mathbb S_1^{n+1}$ and the anti-de Sitter space $\mathbb H_1^{n+1}$, we can obtain a result analogous to Corollary \ref{NS} for spacelike hypersurfaces in these spacetimes.

\begin{corollary}\label{FH} 
Any spacelike hypersurface $M^n$ in a semi-Riemannian warped product space $\varepsilon I\times_a\mathbb Q^n_s(c)$ of constant curvature $\kappa$ satisfies
\begin{align*}
    \dfrac{1}{2}\Delta|\alpha|^2=&|\nabla^\perp\alpha|^2+\sum_{j,k,l}h_{ij}h_{kkji}+n\kappa(S-nH^2)+\delta\Big(nH(trA^3)-S^2\Big). 
\end{align*}
In particular, if $M^n$ is pmc , i.e, the mean curvature vector field $\Vec H$ of $M$ is parallel with respect to normal connetion $\nabla^\perp$, then
\begin{align*}
    \dfrac{1}{2}\Delta|\alpha|^2=&|\nabla^\perp\alpha|^2+n\kappa(S-nH^2)+\delta\Big(nH(trA^3)-S^2\Big).
\end{align*}
\end{corollary}
\begin{proof} 
The statement of result follows from Corollary \ref{C1} and from equation \eqref{constant}.
\end{proof}

\begin{corollary}\label{co3}
Any pmc compact spacelike hypersurface $M^n$ of non-negative sectional curvature in a semi-Riemannian warped product space $\varepsilon I\times_a\mathbb Q^n_s(c)$ of constant curvature is isoparametric. 
\end{corollary}
\begin{proof} Following the ideas of Nomizu and Smyth \cite[page 372]{NS}, 
one can show that if $\{e_1,\ldots,e_n\}$ is an orthonormal basis in $T_{f(x)}M$ for each $x\in M$ such that $Ae_i=\lambda_ie_i, 1\leq i\leq n$, then 
\begin{equation*}
 \kappa(ntrA^2-(trA)^2)+\delta\Big((trA)trA^3-(trA^2)^2\Big)=\sum_{i<j}(\lambda_i-\lambda_j)^2K_{ij},
\end{equation*}
for all $i,j$, where $K_{ij}=(\kappa+\delta\lambda_i\lambda_j), 1\leq i,j\leq n$, is the seccional curvature of $M$. 

Therefore, from Corollary \ref{FH} together with the above identity we have that
    \begin{align*}
        \frac{1}{2}\Delta|\alpha|^2=|\nabla^\perp\alpha|^2+\sum_{i<j}(\lambda_i-\lambda_j)^2K_{ij}.
    \end{align*}
Since $K_{ij}\geq0, \forall i, j$, it follows that $\Delta|\alpha|^2\geq0$. As $M$ is compact we conclude that $\Delta|\alpha|^2=0$ and $|\alpha|^2$ is constant (see \cite[page 215, Theorem 1.3]{Yano}).
    
\end{proof}

We end this section with Simons' type formula for submanifolds in product spaces of type $\mathbb R\times\mathbb Q^n(c)$ that is a particular case of Theorem \ref{T1}.

\begin{corollary}[Fetcu and Rosenberg \cite{DH}]\label{CoRosen}
Let $f:M^n\rightarrow\mathbb R\times\mathbb Q^{n+m}(c)$ be an isometric immersion. Then,
\begin{align*}
 \dfrac{1}{2}\Delta|\alpha|^2=&|\nabla^\perp\alpha|^2+\sum_\beta\sum_{i,j,k}h^\beta_{ij}h^\beta_{kkji}\\
 &+c\sum_\beta\Big((n-||T||^2) trA_\beta^2 -2n||A_\beta T||^2+3trA_\beta\langle A_\beta T,T\rangle\Big)\\
 &-c\sum_\beta(trA_\beta)^2+c\Big(ntrA_\beta^2-(trA_\beta)^2\Big)\\
&+\sum_{\beta,\gamma}\Big((trA_\gamma)tr(A_\beta^2A_\gamma)+tr([A_\gamma,A_\beta]^2)-(tr(A_\beta A_\gamma))^2\Big).  
\end{align*}
In particular, if $f$ is a pmc, then
\begin{align*}
      \dfrac{1}{2}\Delta|\alpha|^2=&|\nabla^\perp\alpha|^2+c\sum_\beta\Big((n-||T||^2) trA_\beta^2 -2n||A_\beta T||^2+3trA_\beta\langle A_\beta T,T\rangle\Big)\\
 &-c\sum_\beta(trA_\beta)^2+c\Big(ntrA_\beta^2-(trA_\beta)^2\Big)\\
&+\sum_{\beta,\gamma}\Big((trA_\gamma)tr(A_\beta^2A_\gamma)+tr([A_\gamma,A_\beta]^2)-(tr(A_\beta A_\gamma))^2\Big).
\end{align*}
\end{corollary}
\begin{proof}
We have that $a=1,\varepsilon=1$ and $s=0$, so $\varepsilon_\beta=1, \forall \beta,$ and $b=-c$. From Theorem \ref{T1} we obtain the result.
\end{proof}

 \section{Pseudo-parallel submanifolds and Simons' type formula}\label{parallel}

We begin by briefly discussing some important concepts 
that will be used henceforth.
 Let $f: M^n\rightarrow\overline M^{n+m}$ be an isometric immersion with second fundamental form $\alpha$. We say that:
\begin{enumerate}
    \item $f$ is totally geodesic if $\alpha(X,Y)=0, \forall X, Y\in TM$;
    \item $f$ is parallel if $(\nabla_Z^\perp\alpha)(X,Y)=0, \forall X, Y, Z\in TM$;
    \item $f$ is semi-parallel if $(\overline R(X,Y)\cdot\alpha)(Z,W)=0, \forall X, Y, Z, W\in TM$;
    \item $f$ is pseudo-parallel if there exist a smooth function $\psi$ on $M^n$ such that \begin{equation}\label{pseudo}
(\overline R(X,Y)\cdot\alpha)(Z,W)=\psi((X\wedge Y)\cdot\alpha)(Z,W), 
\end{equation}
for all $X, Y, Z, W\in TM$,
\end{enumerate}
where 
\begin{align*}
(\overline R(X,Y)\cdot\alpha)(Z,W):=&(\overline\nabla_X\overline\nabla_Y\alpha)(Z,W)-(\overline\nabla_Y\overline\nabla_X\alpha)(Z,W),\\
((X\wedge Y)\cdot\alpha)(Z,W):=&-\alpha((X\wedge Y)Z,W)-\alpha(Z,(X\wedge Y)W),\\
(X\wedge Y)Z:=&\langle Y,Z\rangle X-\langle X,Z\rangle Y.
\end{align*}

 Totally geodesic, parallel and semi-parallel immersion imply that the immersion is pseudo-parallel. However, the converse is generally not true, as demonstrated in  \cite[Example 4.6]{YHLT}. 
 
 We now observe that \eqref{pseudo} is equivalent to
\begin{align}\label{Rperp}\nonumber
    R^\perp(X,Y)(\alpha(Z,W))=&\alpha(R(X,Y)Z,W)+\alpha(Z,R(X,Y)W)\\ \nonumber
    &-\psi\Big(\langle Y,Z\rangle\alpha(X,W)+\langle X,Z\rangle\alpha(Y,W)\Big)\\
    &-\psi\Big(\langle Y,W\rangle\alpha(X,Z)+\langle X,W\rangle\alpha(Y,Z)\Big).
\end{align}

With these concepts in mind we can prove Theorem \ref{c1}, as follows. 

\begin{proof}[Proof of Theorem ~{\upshape\ref{c1}}]

    Let $f:M^n\rightarrow\varepsilon I\times_a\mathbb Q^{n+m}_s(c)$ be a $\psi$-pseudo-parallel spacelike immersion. First of all, we will verify the following identity:  
 \begin{align}\label{pp}\nonumber
  \sum_\beta B\Big(n||A_\beta T||^2\!-\!2\langle A_\beta T,T\rangle trA_\beta\!+\!||T||^2trA_\beta^2\Big)+\sum_{\beta}(\psi+b)&(n trA_\beta^2-(tr A_\beta)^2)\\ 
-\sum_{\beta,\gamma}\varepsilon_\gamma\Big((trA_\gamma)tr(A_\beta^2A_\gamma)+tr([A_\beta,A_\gamma]^2)-(tr(A_\beta A_\gamma))^2\Big)&=0,
 \end{align}
where $B=\frac{a''}{a}-\varepsilon b$.
 In fact, using the notation of section \ref{proof}, we have that the Laplacian of the squared norm of the second fundamental form $\alpha$ satisfies
\begin{align*}
    \dfrac{1}{2}\Delta|\alpha|^2&=|\nabla^\perp\alpha|^2+\sum_\beta\sum_{i,j,k}h^\beta_{ij}h^\beta_{ijkk}.
 \end{align*}
Moreover, 
\begin{align*}
    h^\beta_{ijkk}=&h^\beta_{kjik}+e_k(B)(\langle T,e_i\rangle\delta_{kj}-\langle T,e_k\rangle\delta_{ij})\langle\xi,e_\beta\rangle\\
&+B\frac{a'}{a}(\delta_{ki}\delta_{kj}-\delta_{kk}\delta_{ij})\langle\xi,e_\beta\rangle+2B\frac{a'}{a}\varepsilon\Big(\langle T,e_k\rangle^2\delta_{ij}-\langle T,e_i\rangle\langle T,e_k\rangle\delta_{kj}\Big)\Big)\langle\xi,e_\beta\rangle\\
    &+B(\langle A_{\xi}e_k,e_i\rangle\delta_{kj}-\langle A_\xi e_k,e_k\rangle\delta_{ij})\langle\xi,e_\beta\rangle+B(\langle T,e_k\rangle\delta_{ij}-\langle T,e_i\rangle\delta_{kj})\langle\alpha(e_k,T),e_\beta\rangle.
\end{align*}
 Equation \eqref{pseudo} is equivalent to 
\begin{align*}
h^\beta_{ijkl}=h^\beta_{ijlk}-\psi(\delta_{ki}h^\beta_{lj}-\delta_{li}h^\beta_{kj}+\delta_{kj}h^\beta_{il}-\delta_{lj}h^\beta_{ik}).
\end{align*}
Then,
\begin{align*}
    h^\beta_{ijkk}=&h^\beta_{kjki}-\psi(\delta_{ik}h^\beta_{kj}-\delta_{kk}h^\beta_{ij}+\delta_{ij}h^\beta_{kk}-\delta_{kj}h^\beta_{ik})+e_k(B)(\langle T,e_i\rangle\delta_{kj}-\langle T,e_k\rangle\delta_{ij})\langle\xi,e_\beta\rangle\\
&+B\frac{a'}{a}(\delta_{ki}\delta_{kj}-\delta_{kk}\delta_{ij})\langle\xi,e_\beta\rangle+2B\frac{a'}{a}\varepsilon\Big(\langle T,e_k\rangle^2\delta_{ij}-\langle T,e_i\rangle\langle T,e_k\rangle\delta_{kj}\Big)\Big)\langle\xi,e_\beta\rangle\\
    &+B(\langle A_{\xi}e_k,e_i\rangle\delta_{kj}-\langle A_\xi e_k,e_k\rangle\delta_{ij})\langle\xi,e_\beta\rangle+B(\langle T,e_k\rangle\delta_{ij}-\langle T,e_i\rangle\delta_{kj})\langle\alpha(e_k,T),e_\beta\rangle.
\end{align*}
Using that $h^\beta_{kjki}=h^\beta_{jkki}$ and Codazzi equation, we have
\begin{align*}
    h^\beta_{jkki}=&h^\beta_{kkji}+e_i(B)(\langle T,e_j\rangle\delta_{kk}-\langle T,e_k\rangle\delta_{jk})\langle\xi,e_\beta\rangle\\
    &+B\frac{a'}{a}(\delta_{ij}\delta_{kk}-\delta_{ik}\delta_{jk})\langle\xi,e_\beta\rangle+2B\frac{a'}{a}\varepsilon\Big(\langle T,e_k\rangle\langle T,e_i\rangle\delta_{jk}-\langle T,e_i\rangle\langle T,e_j\rangle\delta_{kk}\Big)\langle\xi,e_\beta\rangle\\
    &+B\Big(\langle A_\xi e_i,e_j\rangle\delta_{kk}-\langle A_\xi e_i,e_k\rangle\delta_{jk}\Big)\langle\xi,e_\beta\rangle+B\Big(\langle T,e_k\rangle\delta_{jk}-\langle T,e_j\rangle\delta_{kk}\Big)\langle \alpha(e_i,T),e_\beta\rangle,
\end{align*}
which implies that

\begin{align*}
    h^\beta_{ijkk}=&h^\beta_{kkji}-\psi(\delta_{ik}h^\beta_{kj}-\delta_{kk}h^\beta_{ij}+\delta_{ij}h^\beta_{kk}-\delta_{kj}h^\beta_{ik})\\
    &+e_k(B)\Big(\langle T,e_i\rangle\delta_{kj}-\langle T,e_k\rangle\delta_{ij}\Big)\langle\xi,e_\beta\rangle\\
    &+e_i(B)\Big(\langle T,e_j\rangle\delta_{kk}-\langle T,e_k\rangle\delta_{jk}\Big)\langle\xi,e_\beta\rangle\\
    &+2B\frac{a'}{a}\varepsilon\Big(\langle T,e_k\rangle^2\delta_{ij}-\langle T,e_i\rangle\langle T,e_j\rangle\delta_{kk}\Big)\langle\xi,e_\beta\rangle\\
    &+B\Big(\langle A_\xi e_i,e_j\rangle\delta_{kk}-\langle A_\xi e_k,e_k\rangle\delta_{ij}\Big)\langle\xi,e_\beta\rangle\\
    &+B\Big(\langle T,e_k\rangle\delta_{ij}-\langle T,e_i\rangle\delta_{kj}\Big)\langle\alpha(e_k,T),e_\beta\rangle\\
    &+B\Big(\langle T,e_k\rangle\delta_{jk}-\langle T,e_j\rangle\delta_{kk}\Big)\langle \alpha(e_i,T),e_\beta\rangle.
\end{align*}
This way, 
\begin{align*}
    \sum_{\beta}\sum_{i,j,k}h_{ij}^\beta h_{ijkk}^\beta=& \sum_{\beta}\sum_{i,j,k}h_{ij}^\beta h_{kkji}^\beta+\psi(n trA_\beta^2-(tr A_\beta)^2)\\
    &+\sum_{\beta}\Big\langle\nabla \Big(\frac{a''}{a}-\varepsilon b\Big),(n A_\beta-(tr A_\beta)I)(T)\Big\rangle\langle\xi,e_\beta\rangle\\
    &-b'\sum_{\beta}\Big\langle(n A_\beta-(tr A_\beta)I)(T), T\Big\rangle\langle\xi,e_\beta\rangle\\
    &+\Big(\frac{a''}{a}-\varepsilon b\Big)\sum_{\beta}\Big(tr (A_\beta A_\xi)-(trA_\beta)(trA_\xi)\Big)\langle\xi,e_\beta\rangle\\
    &+\Big(\frac{a''}{a}-\varepsilon b\Big)\sum_{\beta}\Big(\langle A_\beta T,T\rangle trA_\beta-n||A_\beta T||^2\Big)
\end{align*}
and so
\begin{align*}
    \frac{1}{2}\Delta|\alpha|^2=&|\nabla^\perp\alpha|^2+\sum_{\beta}\sum_{i,j,k}h_{ij}^\beta h_{kkji}^\beta+\psi(n trA_\beta^2-(tr A_\beta)^2)\\
    &+\sum_{\beta}\Big\langle\nabla \Big(\frac{a''}{a}-\varepsilon b\Big),(n A_\beta-(tr A_\beta)I)(T)\Big\rangle\langle\xi,e_\beta\rangle\\
    &-b'\sum_{\beta}\Big\langle(n A_\beta-(tr A_\beta)I)(T), T\Big\rangle\langle\xi,e_\beta\rangle\\
    &+\Big(\frac{a''}{a}-\varepsilon b\Big)\sum_{\beta}\Big(tr (A_\beta A_\xi)-(trA_\beta)(trA_\xi)\Big)\langle\xi,e_\beta\rangle\\
    &+\Big(\frac{a''}{a}-\varepsilon b\Big)\sum_{\beta}\Big(\langle A_\beta T,T\rangle trA_\beta-n||A_\beta T||^2\Big).
\end{align*}
The last identity together with Theorem \ref{T1} prove \eqref{pp}.

Now, we assume that $f$ is extremal. It follows from \eqref{pp} that
\begin{align*}
nB\sum_\beta||A_{\beta}T||^2+\Big(B||T||^2+n(\psi+b)\Big)\sum_\beta trA_{\beta}^2
+\sum_{\beta,\gamma}\varepsilon_\gamma\Big(|[A_\beta,A_\gamma]|^2+(tr(A_\beta A_\gamma))^2\Big)=0,
\end{align*}
where we use that $[A_{\beta},A_{\gamma}]^{\textbf{t}}=-[A_{\beta},A_{\gamma}]$, $\textbf{t}$ indicates the transpose of linear operator.

We note that, for $s=0$ and $\varepsilon=1$, we get $\varepsilon_{n+1}=\varepsilon_{n+2}=\dots=\varepsilon_{n+m+1}=1$. Therefore,  $\alpha=0$ if  
\begin{align*}
 B=\Big(\frac{a''}{a}-\frac{(a')^2-\varepsilon c}{a^2}\Big)\geq0
\end{align*} 
 and 
 \begin{align*}
 \Big(\Big(\frac{a''}{a}-\frac{(a')^2-\varepsilon c}{a^2}\Big)||T||^2+n\Big(\psi+\frac{\varepsilon(a')^2-c}{a^2}\Big)\Big)\geq0.
\end{align*}
Analogously, for $s=0$ and $\varepsilon=-1$ or $0<s=m+\frac{1+\varepsilon}{2}$ with $\varepsilon=\pm 1$, the result follows with the inequalities reversed, since in this case $\varepsilon_{n+1}=\varepsilon_{n+2}=\dots=\varepsilon_{n+m+1}=-1$.
\end{proof}
\begin{remark}{\upshape{
When $A_\beta T=0$, for all $\beta$, in particular for $T=0$, the condition regarding the function $\frac{a''}{a}-\frac{(a')^2}{a^2}+\frac{\varepsilon c}{a^2}$ is not required to obtain the result of Theorem \ref{c1}.}}   
\end{remark}

The next three examples show that the assumptions of Theorem \ref{c1} are indeed necessary.
\begin{eexample}[Riemannian case: $s=0$ and $\varepsilon=1$]\label{ex1}\upshape{
     Let $\varphi:\mathbb S^3\Big(\frac{1}{3}\Big)\rightarrow\mathbb S^4(1)$ be the isometric immersion defined by
        \begin{align*}
            \varphi(x,y,z)=\Big(\dfrac{1}{\sqrt{3}}xy,\dfrac{1}{\sqrt{3}}xz,\dfrac{1}{\sqrt{3}}yz,\dfrac{1}{2\sqrt{3}}(x^2-y^2),\dfrac{1}{6}(x^2+y^2-2z^2)\Big).
        \end{align*}
        The immersion $\varphi$ is called {\textit{Veronese surface}}. One can check that $\varphi$ is extremal, locally parallel ($\psi=0$) and non-totally geodesic, see \cite{Chern,saka}. Accordingly to \cite{YHLT}, we consider $f=i\circ\varphi:\mathbb S^3\Big(\frac{1}{3}\Big)\rightarrow\mathbb R\times\mathbb S^4(1)$, where $i:\mathbb S^4(1)\rightarrow\mathbb R\times\mathbb S^4(1)$ is a totally geodesic inclusion and so $f$ is pseudo-parallel and non-totally geodesic. Moreover, $f$ satisfies $\frac{a''}{a}-\frac{(a')^2}{a^2}+\dfrac{1}{a^2}>0$ and  \begin{align*}
    -\frac{1}{n}\Big(\Big(\frac{a''}{a}-\frac{(a')^2}{a^2}+\frac{c}{a^2}\Big)||T||^2+n\Big(\frac{(a')^2- c}{a^2}\Big)\Big)=1-\dfrac{||T||^2}{3}>0=\psi.\end{align*}}  
\end{eexample}
Example \ref{ex1} was studied in \cite{YHLT} to prove that for $n>3$ there are examples of pseudo-parallel immersions with non flat normal bundle.  
\begin{eexample}[Case: $s=0$ and $\varepsilon=-1$ (Robertson-Walker space)]\upshape{ Let $\mathbb H_1^{n+1}(-1)$ be the spacetime obtained into Minkowski space $\mathbb R_2^{n+2}$ as the hyperquadric $\mathbb H^{n+1}_1(-1)=\{x\in \mathbb R_2^{n+2};\langle x,x\rangle=-1\}$. This space has constant curvature $-1$ and it is called anti-de Sitter space. Moreover, $\mathbb H_1^{n+1}(-1)$ can be described as the Lorentzian warped product $(-\frac{\pi}{2},\frac{\pi}{2})\times_{\cos t}\mathbb H^n(-1)$ whose the warped product metric is given by $\langle,\rangle=-dt^2+\cos^2(t)g_{\mathbb H^n}$, where $\mathbb H^n(-1)$ denotes the hyperbolic space of constant curvature $-1$. For more detail see \cite{montiel}. We consider $f:\mathbb H^{m}\Big(-\frac{n}{m}\Big)\times\mathbb H^{n-m}\Big(-\frac{n}{n-m}\Big)\rightarrow \mathbb H_{1}^{n+1}(-1), 1\leq m\leq n-1,$ the complete extremal isometric immersion described in \cite{TI}. Furthermore, it was still proved that the norm square of the second fundamental form of $f$ is $n$. Then, from Corollary \ref{FH} we have that $f$ is parallel immersion, since
\begin{align*}
    0= \dfrac{1}{2}\Delta|\alpha|^2=|\nabla^\perp\alpha|^2-(-1)(-n^2)+(-1)(-n^2).
\end{align*}
We notice that $\dfrac{a''}{a}-\dfrac{(a')^2}{a^2}+\dfrac{\varepsilon c}{a^2}=0$ and 
\begin{align*}
    -\frac{1}{n}\Big(\Big(\frac{a''}{a}-\frac{(a')^2}{a^2}+\frac{\varepsilon c}{a^2}\Big)||T||^2+n\Big(\frac{\varepsilon(a')^2- c}{a^2}\Big)\Big)=-1<0=\psi.\end{align*} }   
\end{eexample}

We will now demonstrate the following Lemma to be applied in the last example.
\begin{lemma}\label{cylinder}
    Let $f:M^n\rightarrow\mathbb Q_s^{n+m}(c)$ be a $\psi$-pseudo-parallel spacelike immersion with vanishing mean curvature vector. Then, $F:I\times M^n\rightarrow I\times\mathbb Q_s^{n+m}(c)$ defined by $F(s,x)=(s,f(x))$ is a $\psi$-pseudo parallel spacelike immersion which has mean curvature vanishes identically. 
  \end{lemma}  \begin{proof}
    Of course, $F_{\ast}(\frac{\partial}{\partial s})=\frac{\partial}{\partial t}=T$ and $F_{\ast}(X)=f_\ast(X)$ for each $X\in TM$. Then, the normal vector fields to immersion $F$ are the normal vector fields to immersion $f$. Moreover, 
\begin{align*}
    \alpha_F(X,Y)=\alpha_f(X,Y), \forall X, Y\in TM, 
\end{align*}
where $\alpha_F$ and $\alpha_f$ denote the second fundamental forms of $F$ and $f$, respectively. As $\overline\nabla_TT=0$, in particular, $\alpha_F(T,T)=0$, it follows that $\Vec{H}_F=\frac{n}{n+1}\Vec{H}_f=0$, where $\Vec{H}_F$ and $\Vec{H}_f$ denote the mean curvature vector of $F$ and $f$, respectively. With these identity in mind and noting that $\overline\nabla_TX=\overline\nabla_XT=0, \forall X\in TM$, we can conclude that 
\begin{align*}
    (\overline R(X,Y)\cdot\alpha)(Z,W)=&\psi((X\wedge Y)\cdot\alpha)(Z,W),\\
    (\overline R(X,T)\cdot\alpha)(Z,W)=&\psi((X\wedge T)\cdot\alpha)(Z,W)=0,\\
    (\overline R(X,Y)\cdot\alpha)(T,W)=&\psi((X\wedge Y)\cdot\alpha)(T,W)=0,
\end{align*}
for each $X,Y,Z,W\in TM$. Therefore, we obtain the desired results.
\end{proof}

\begin{eexample}[Case: $0<s=m+\frac{1+\varepsilon}{2}$]\upshape{ 
Let $\varphi:\mathbb H^2(-\frac{1}{3})\rightarrow\mathbb H^4_2(-1)$ be the spacelike isometric immersion defined by 
\begin{align*}
    \varphi(u,v)=&\Big(\frac{\sqrt{3}}{2}\sinh\Big(\frac{2v}{\sqrt{3}}\Big)\sin\Big(\frac{u}{\sqrt{3}}\Big),\frac{\sqrt{3}}{2}\sinh\Big(\frac{2v}{\sqrt{3}}\Big)\cos\Big(\frac{u}{\sqrt{3}}\Big),\frac{\sqrt{3}}{2}\sinh^2\Big(\frac{v}{\sqrt{3}}\Big)\sin\Big(\frac{2u}{\sqrt{3}}\Big),\\
    &\frac{\sqrt{3}}{2}\sinh^2\Big(\frac{v}{\sqrt{3}}\Big)\cos\Big(\frac{2u}{\sqrt{3}}\Big),\frac{1}{2}\Big(3\cosh^2\Big(\frac{v}{\sqrt{3}}\Big)-1\Big)\Big),
\end{align*}
  where $\mathbb H^2\Big(-\frac{1}{3}\Big)$ is the Riemannian hyperbolic space of constant curvature $-\frac{1}{3}$ and $\mathbb H^4_2(-1)$ is the semi-Riemannian hyperbolic space of constant curvature $-1$ with index $2$. The immersion $\varphi$ is called the \textit{hyperbolic Veronese surface}. One can show easily that $\varphi$ is extremal, parallel ($\psi=0$) and non-totally geodesic, see \cite{Se}. From Lemma \ref{cylinder}, the function $f:\mathbb R\times\mathbb H^2\Big(-\frac{1}{3}\Big)\rightarrow \mathbb R\times\mathbb H_2^4(-1)$ defined by $f(s,u,v)=(s,\varphi(u,v))$ is an extremal spacelike immersion which is parallel and non-totally geodesic, since $\alpha_f=\alpha_\varphi$. Now, we note that
 \[\frac{a''}{a}-\frac{(a')^2}{a^2}+\frac{\varepsilon c}{a^2}=-1<0\] and
   \begin{align*}
   -\frac{1}{n}\Big(\Big(\frac{a''}{a}-\frac{(a')^2}{a^2}+\frac{\varepsilon c}{a^2}\Big)||T||^2+n\Big(\frac{\varepsilon(a')^2- c}{a^2}\Big)\Big)=&-\frac{1}{3}(3-||T||^2)=-\frac{2}{3}<0=\psi,\end{align*} where we used that $||T||^2=\langle\partial_t,\partial_t\rangle=1$.
    }
\end{eexample}
The following corollary is an immediate consequence of Theorem \ref{c1}. 
\begin{corollary}\label{deSitter}
    Let $f:M^n\rightarrow\mathbb S_1^{n+1}(\kappa)$ be a semi-parallel spacelike immersion. If $f$ is extremal, then $f$ is totally geodesic.
\end{corollary}
\begin{proof}
  Let us consider $-\mathbb R\times_{\frac{\cosh(\sqrt{\kappa} t)}{\sqrt{\kappa}}}\mathbb S^n(1)$ as one of the models the Lorentz warped product of the de Sitter space $\mathbb S_1^{n+1}(\kappa)$, where $\mathbb S^n$ denotes the unit sphere. Then, \[\dfrac{a''}{a}-\dfrac{(a')^2}{a^2}+\dfrac{\varepsilon c}{a^2}=\dfrac{\sqrt{k}\cosh(\sqrt{k}t)}{\frac{\cosh(\sqrt{k}t)}{\sqrt{k}}}-\dfrac{\sinh^2(\sqrt{k}t)}{\frac{\cosh^2(\sqrt{k}t)}{k}}-\dfrac{1}{\frac{\cosh^2(\sqrt{k}t)}{k}}=0,\] and 
    \begin{align*}  
    -\frac{1}{n}\Big(\Big(\frac{a''}{a}-\frac{(a')^2}{a^2}+\frac{\varepsilon c}{a^2}\Big)||T||^2+n\Big(\frac{\varepsilon(a')^2- c}{a^2}\Big)\Big)= \Big(\frac{\sinh^2({\sqrt{\kappa t}})+ 1}{\frac{\cosh^2({\sqrt{\kappa t}})}{\kappa}}\Big)=\kappa>0=\psi.
    \end{align*}
    It follows from Theorem \ref{c1} that $f$ is totally geodesic.
\end{proof}

Another consequence of Theorem \ref{c1} is related to a result concerning the nonexistence of extremal semi-parallel spacelike hypersurfaces in the Einstein-de Sitter spacetime. Here, we denote $\mathcal{E}^{n+1}$ as the $n+1$-dimensional Einstein-de Sitter spacetime, characterized as the Lorentzian warped product $(0,\infty)\times_{t^{\frac{1}{3}}}\mathbb R^n$ equipped with the warped metric $-dt^2+t^{\frac{2}{3}}g_{\mathbb R^n}$, where $g_{\mathbb R^n}$ denotes the Euclidean metric of constant sectional curvature $0$. This spacetime is
a particular example of a Robertson-Walker spacetime with non-constant curvature which represents a classical model for an expanding universe. 

To establish the next result we reference the following result proved by Aledo et. al \cite[Corollary 14]{Aledo}:
``\textit{There are no totally geodesic spacelike hypersurfaces in the Einstein-de Sitter spacetime}''.

\begin{corollary}\label{Einstein}
    There are no extremal semi-parallel spacelike hypersurfaces in the Einstein-de Sitter spacetime $\mathcal{E}^{n+1}$.
\end{corollary}
\begin{proof}
    Let us suppose that there exists $M^n$, an extremal semi-parallel spacelike hypersurface in $\mathcal{E}^{n+1}$. Then, we see that
    \begin{align*}
      \frac{a''}{a}-\frac{(a')^2}{a^2}+\frac{\varepsilon c}{a^2}=-\dfrac{1}{3}t^{-2}<0,
  \end{align*}
  and
  \begin{align*}
    -\frac{1}{n}\Big(\Big(\frac{a''}{a}-\frac{(a')^2}{a^2}+\frac{\varepsilon c}{a^2}\Big)||T||^2+n\Big(\frac{\varepsilon(a')^2- c}{a^2}\Big)\Big)=\dfrac{t^{-2}}{3n}\Big(||T||^2+\frac{n}{3}\Big)>0=\psi.   
  \end{align*}
  Since $f$ is extremal, it follows from Theorem \ref{c1} that $f$ is totally geodesic, which is not possible for spacelike hypersurface in $\mathcal{E}^{n+1}$, as shown by \cite{Aledo}.
\end{proof}

To prove the last corollary of this paper we require the following Lemma.
\begin{lemma}\label{R}
     Let $f:M^n\rightarrow \varepsilon I\times_a\mathbb Q_s^{n+m}(c)$ be a $\psi$-pseudo-parallel spacelike immersion. Then, $R^\perp(X,Y)\Vec{H}=0$ for each $X,Y\in TM$.
\end{lemma}
\begin{proof}
    Our proof is analogous to \cite{AS} and we will do it here for completeness. Let $\eta\in TM^\perp$ and $\{X_1,\ldots,X_n\}$ be an orthonormal basis in $T_xM$ for each $x\in M$ such that $A_\eta X_i=\lambda_iX_i, 1\leq i\leq n$. Then, 
    \begin{align*}
        \langle R^\perp(X,Y)\Vec{H},\eta\rangle=&\frac{1}{n}\sum_{i}\langle R^\perp(X,Y)\alpha(X_i,X_i),\eta\rangle\\
        =&\frac{2}{n}\sum_{i}\langle \alpha(R(X,Y)X_i,X_i),\eta\rangle\\
        =&\frac{2}{n}\sum_{i}\langle A_\eta(X_i),R(X,Y)X_i\rangle=\frac{2}{n}\sum_{i}\lambda_i\langle X_i,R(X,Y)X_i\rangle=0,
    \end{align*}
    where we use the identity \eqref{Rperp}.
\end{proof}

\begin{corollary}\label{cof}
    Let $f:M^n\rightarrow \varepsilon I\times_a\mathbb Q_s^{n+1}(c)$ be a $\psi$-pseudo-parallel spacelike immersion. If either the assumptions of Theorem \ref{c1} hold or $\Vec H(p)\neq0$ for some $p\in M$, then $f$ has flat normal bundle.
\end{corollary}
\begin{proof} If the hypothesis from Theorem \ref{c1} are satisfied, there is nothing to prove. So, we suppose that
    $\Vec{H}(p)\neq0$ for some $p\in M$. Then, Lemma \ref{R} implies that $\Vec{H}(p)$ belongs to Kernel of Linear operator $R^\perp(X,Y):T_pM^{\perp}\rightarrow T_pM^\perp$. Since $R^\perp(X,Y)$ is skew-symmetric and $T_pM^\perp$ has dimension $2$, it follows that $R^\perp(X,Y)=0,\forall X, Y\in T_pM$. 
\end{proof}

As a final highlight of this article, we would like to pose the following questions that arise naturally from the consequences of Theorem 2. 
These questions represent significant aspects in the investigation of the classification of pseudo-parallel spacelike submanifolds in $\varepsilon I\times_a\mathbb Q_s^{n+m}(c)$.
\begin{itemize}
\item [ 1.] Is it possible to find semi-parallel spacelike hypersurfaces with non-zero mean curvature in the Einstein-de Sitter spacetime? More generally, can pseudo-parallel spacelike hypersurfaces with non-zero mean curvature exist in the Einstein-de Sitter spacetime?
\item [2.]In their work \cite{LT}, Lobos and Tojeiro provided a complete local classification of pseudo-parallel submanifolds with flat normal bundles in space forms. Taking into account of this article and Corollary \ref{cof}, one can ask:
Is it possible to formulate a similar, albeit more general, result pertaining to the local classification of pseudo-parallel submanifolds with flat normal bundles in $\varepsilon I\times_a\mathbb Q^{n+m}_s(c)$?
\end{itemize}






\end{document}